\documentclass[12pt,a4paper]{article}
\usepackage[left=2cm, right=2cm, top=2.5cm, bottom=2.5cm]{geometry}
\usepackage{enumerate}
\usepackage{amssymb}
\usepackage{amsmath}
\usepackage{mathrsfs}
\usepackage{amsthm}
\usepackage{mathtools}
\usepackage{enumitem}
\theoremstyle{definition}

\newtheorem{theorem}{Theorem}[section]
\newtheorem{lemma}[theorem]{Lemma}

\newtheorem{corollary}[theorem]{Corollary}
\newtheorem{remark}[theorem]{Remark}

\usepackage[colorlinks,linkcolor=blue]{hyperref}
\numberwithin{equation}{section}

\usepackage{enumerate}

\allowdisplaybreaks[4]

\makeatletter

\newcommand{\Rmnum}[1]{\expandafter\@slowromancap\romannumeral #1@}
\makeatother

\begin{document}
	
	\title{\uppercase{Michael-Simon Sobolev inequalities 
			in Euclidean space 
	}}
	
	\author{
		Yuting Wu 
		\thanks{School of Mathematical Sciences, East China Normal University, 500 Dongchuan Road, Shanghai 200241,
			P. R. of China, E-mail address: 52215500001@stu.ecnu.edu.cn. }
		\and
		Chengyang Yi 
		\thanks{School of Mathematical Sciences, East China Normal University, 500 Dongchuan Road, Shanghai 200241,
			P. R. of China, E-mail address: 52195500013@stu.ecnu.edu.cn. }	
	}
	
	\date{}
	
	\maketitle
	
	\begin{abstract}
		Inspired by \cite{Bre, DP}, we prove Michael-Simon type inequalities for smooth symmetric uniformly positive define $\left( 0,2\right)$-tensor fields on compact submanifolds in Euclidean space by the Alexandrov-Bakelman-Pucci (ABP) method.
	\end{abstract}
	
	\section{Introduction}
	
	Many authors have been studied extensive researches about optimal transport techniques \cite{CV, BE, Wang, KM} and the Alexandrov-Bakelman-Pucci (ABP) method \cite{Bren, YW, Dong22, Joh} for inequalities. 
	We remark that optimal transport techniques also apply to more general metric measure spaces (see \cite{CM}).
    Meanwhile, the ABP method has been considered in many literatures, 
	such as Brendle \cite{Bre} proved a Sobolev inequality for submanifolds of arbitrary dimension and codimension in Euclidean space by the ABP method in 2019;
    Soon, Brendle \cite{Bren} obtained a Sobolev inequality in Riemannian manifolds with nonnegative curvature;
    In recent years, a large number of experts \cite{LL, Wang, Joh, Dong22, Ma, DP} have made further investigations about the Sobolev inequality.
      
    In 2018, D. Serre \cite{DS} proved a Michael-Simon Sobolev inequality involving a positive symmetric matrix-valued function $A$ on a smooth bounded convex domain in $\mathbb{R}^{n}$ and showed some applications to fluid dynamics by optimal transport techniques. D. Pham \cite{DP} reorganized some of the conclusions of \cite{DS}. The open unit ball in $\mathbb{R}^{m}$ is denoted by $B^{m}$.	
    \begin{theorem}	\label{th:A}
    	($\cite{DS}$, Theorem 2.3 and Proposition 2.2
    	or $\cite{DP}$, Theorem 1.7)
    	Let $n\in\mathbb{N}$ and 
    	$\Omega$ be a smooth bounded convex domain in $\mathbb{R}^{n}$. 
    	If $A$ is a smooth uniformly positive symmetric matrix-valued function on $\overline{\Omega}$
    	(the closure of $\Omega$), then
    	\begin{equation*}   \label{eq:A}
    		\int_{\Omega}\left|\mathrm{div}A\left( x\right)  \right| dx
    		+\int_{\partial\Omega}\left| A\left( \nu\right) \right|d\sigma\left( x\right) 
    		\geq n\left| B^{n}\right|^{\frac{1}{n}} 
    		\left( \int_{\Omega}\mathrm{det}\left( A\left( x\right) \right)^{\frac{1}{n-1}}dx\right)  ^{\frac{n-1}{n}},
    	\end{equation*}
    	where $\nu$ is the unit outer normal vector field on $\partial\Omega$. The equality holds if and only if $A$ is divergence-free, there is a smooth convex function $u$ on $\Omega$
    	such that $\nabla u\left( \Omega\right)$ is a ball centered at the origin and 
    	$A=\left(\mathrm{cof} D^{2}u\right)$, 
    	the cofactor matrix of $D^{2}u$.
    \end{theorem}
    Until recently D. Pham \cite{DP} extended Theorem 1.1 to a smooth bounded domain in $\mathbb{R}^{n}$ without the convexity condition of $\Omega$ by the ABP method \cite{Bre, Cab, Cabr}.
    \begin{theorem}	\label{th:B}
    	($\cite{DP}$, Theorem 1.8)
    	Let $n\in\mathbb{N}$ and 
    	$\Omega$ be a smooth bounded domain in $\mathbb{R}^{n}$. 
    	If $A$ is a smooth uniformly positive symmetric matrix-valued function on $\overline{\Omega}$, then
    	\begin{equation*}   \label{eq:B}
    		\int_{\Omega}\left|\mathrm{div}A\left( x\right)  \right| dx
    		+\int_{\partial\Omega}\left| A\left( \nu\right) \right|d\sigma\left( x\right) 
    		\geq n\left| B^{n}\right|^{\frac{1}{n}} 
    		\left( \int_{\Omega}\mathrm{det}\left( A\left( x\right) \right)^{\frac{1}{n-1}}dx\right)  ^{\frac{n-1}{n}},
    	\end{equation*}
    	where $\nu$ is the unit outer normal vector field on $\partial\Omega$. Moerover, 
    	the equality holds if and only if $A$ is divergence-free, there is a smooth convex function $u$ on $\Omega$
    	such that $\nabla u\left( \Omega\right)$ is a ball centered at the origin and 
    	$A=\left(\mathrm{cof} D^{2}u\right)$, 
    	the cofactor matrix of $D^{2}u$.
    \end{theorem}
    In this paper, we prove several Michael-Simon-Sobolev inequalities for smooth symmetric uniformly positive define $\left( 0,2\right)$-tensor fields on compact submanifolds in Euclidean space. Our main result is the following theorem.
	\begin{theorem}	\label{th:01}
		Let $\Sigma^{n}$ be a compact $n$-dimensional submanifold of $\mathbb{R}^{n+m}$ with smooth boundary $\partial\Sigma$ (possibly  $\partial\Sigma=\varnothing$), 
		where $m\geq2$. 
		If $A$ is a smooth symmetric uniformly positive define $\left( 0,2\right)-$tensor field on $\Sigma$,  then
		\begin{equation*}
			\int_{\Sigma}\sqrt{\left| \mathrm{div}_{\Sigma}A\right|^{2}+\left| \left\langle A,\Rmnum{2}\right\rangle \right|^{2}}
			+\int_{\partial\Sigma}\left| A\left( \nu\right) \right| \geq
			n\left[ \frac{\left( n+m\right)\left| B^{n+m}\right|}{m\left| B^{m}\right|}\right] ^{\frac{1}{n}}\left( \int_{\Sigma}\left( \mathrm{det}A\right)^{\frac{1}{n-1}}\right) ^{\frac{n-1}{n}},
		\end{equation*}
		where $\nu$ is the unit outer conormal vector field on $\partial\Sigma$, $\Rmnum{2}$ is the second fundamental form of $\Sigma$.
	\end{theorem}
    When $m=2$,
    because of $\left( n+2\right)\left| B^{n+2}\right|=2\left| B^{2}\right| \left| B^{n}\right|$,
    we obtain a sharp Sobolev inequality for submanifolds of codimensional 2.  
 	\begin{theorem}
 	Let $\Sigma^{n}$ be a compact $n$-dimensional submanifold of $\mathbb{R}^{n+2}$ with smooth boundary $\partial\Sigma$ (possibly  $\partial\Sigma=\varnothing$),  
 	and let $A$ be a smooth symmetric uniformly positive define $\left( 0,2\right)$-tensor field on $\Sigma$, then
 	\begin{equation*}
 		\int_{\Sigma}\sqrt{\left| \mathrm{div}_{\Sigma}A\right|^{2}+\left| \left\langle A,\Rmnum{2}\right\rangle \right|^{2}}
 		+\int_{\partial\Sigma}\left| A\left( \nu\right) \right| \geq
 		n\left| B^{n}\right|^{\frac{1}{n}}\left( \int_{\Sigma}\left( \mathrm{det}A\right)^{\frac{1}{n-1}}\right) ^{\frac{n-1}{n}}.
 	\end{equation*}
 The equality holds if and only if  
  $\Sigma$ is a compact domain in $\mathbb{R}^{n}\subset \mathbb{R}^{n+2}$, $A$ is divergence-free, there is a smooth convex function $u$ on $\Sigma$
  such that $\nabla^{\Sigma} u\left( \Sigma\right)$ is a unit closed ball in $\mathbb{R}^{n}$ centered at the origin and 
  $A=\mathrm{cof} D^{2}_{\Sigma}u$, 
  the cofactor tensor of $D^{2}_{\Sigma}u$. 
 \end{theorem}   
    Since $\mathbb{R}^{n+1}$ is a totally geodesic submanifold in $\mathbb{R}^{n+2}$, Theorem 1.3 and Theorem 1.4 imply a sharp Sobolev inequality for submanifolds of codimension 1:

	\begin{corollary}	\label{co:02}
	Let $\Sigma^{n}$ be a compact $n$-dimensional submanifold of $\mathbb{R}^{n+1}$ with smooth boundary $\partial\Sigma$ (possibly  $\partial\Sigma=\varnothing$), 
	if $A$ is a smooth symmetric uniformly positive define $\left( 0,2\right)$-tensor field on $\Sigma$,  then
	\begin{equation*}
		\int_{\Sigma}\sqrt{\left| \mathrm{div}_{\Sigma}A\right|^{2}+\left| \left\langle A,\Rmnum{2}\right\rangle \right|^{2}}
		+\int_{\partial\Sigma}\left| A\left( \nu\right) \right| \geq
		n\left| B^{n}\right|^{\frac{1}{n}}
		\left( \int_{\Sigma}\left( \mathrm{det}A\right)^{\frac{1}{n-1}}\right) ^{\frac{n-1}{n}},
	\end{equation*}
	where $\nu$ is the unit outer conormal vector field on $\partial\Sigma$, $\Rmnum{2}$ is the second fundamental form of $\Sigma$.
    \end{corollary}
    \begin{remark}	\label{Re:02}
    In the special case of when $A$ is a conformal metric on $\Sigma$, i. e. $A=fg_{\Sigma}$ for a positive function $f$ on $\Sigma$, we have $|\mathrm{div}_{\Sigma}A|^{2}=|\nabla^{\Sigma}f|^{2}$,  $\left\langle A,\Rmnum{2}\right\rangle=fH$ and $\left| A\left( \nu\right) \right| =f$, where $H$ denotes the mean curvature vector of $\Sigma$. Therefore, the above inequalities now are exactly the results in \cite{Bre}.
    \end{remark}
    This paper is organized as follows. In Section 2, we will give basic notions and a generalized trace inequality for the product of two square matrices. In Section 3 and 4, we will give the proofs of Theorem 1.3 and Theorem 1.4, respectively.

    {\bf Acknowledgement.}  This work was partially supported by Science and Technology Commission of Shanghai Municipality (No. 22DZ2229014) and the National Natural Science Foundation of China (Grant No. 12271163). The research is supported by Shanghai Key Laboratory of PMMP.

	\section{Preliminaries}
    
    Let $\Sigma^{n}$ be a compact $n$-dimensional submanifold	with smooth boundary $\partial\Sigma$ (possibly  $\partial\Sigma=\varnothing$) in Eucldiean space.
    In addition, $g_{\Sigma}$ denotes the induced Riemannian metric on $\Sigma$,
    $D_{\Sigma}$ denotes the Levi-Civita connection on $\Sigma$, 
    and $\nabla^{\Sigma}$ denotes the gradient of $\Sigma$.
    Let $A$ is a smooth symmetric uniformly positive defined $\left( 0,2\right)-$tensor field on $\Sigma$.
    For each point $x\in\Sigma$, we denote by $T_{x}\Sigma$ and $T_{x}^{\bot} \Sigma$ the tangent and normal space to  $\Sigma$ at $x$, respectively.
    Let $\left( x^{1}, \dots, x^{n}\right) $ be a local coordinate system on $\Sigma$,
    the divergence of A on $\Sigma$ is defined by
    \begin{equation*}
    	\mathrm{div}_{\Sigma}A:=
    	g_{\Sigma}^{ki}D^{\Sigma}_{k}A_{ij}dx^{j}.
    \end{equation*}
    Let $T$ and $S$ are two $\left(0,2\right)$-tensor fields on $\Sigma$. In general, the inner product of $T$ and $S$ can be written as
    \begin{equation*}
    	\left\langle T, S\right\rangle 
    	=g_{\Sigma}^{ik}g_{\Sigma}^{jl}T_{ij}S_{kl}
    	=T_{ij}S^{ij}.
    \end{equation*}
The composition of $T$ and $S$ is the $\left(0,2\right)$-tensor $T\circ S$ defined by 
    \begin{equation*}
(T\circ S)_{ij}=g^{kl}_{\Sigma}T_{ik}S_{lj}.
\end{equation*}
The determinant of $\mathrm{det}T$ is defined by the determinant of $\left(1,1\right)$-tensor $g_{\Sigma}^{ik}T_{jk}\frac{\partial}{\partial x^{i}}\otimes dx^{j}$. When $S>0$ and $T\circ S=\mathrm{det}S\cdot g_{\Sigma}$, we call $T$ the cofactor tensor of $S$. When $T\circ S=g_{\Sigma}$, we call $T$ the inverse tensor of $S$ denoted by $T^{-1}$. Meanwhile, $\Rmnum{2}$ denotes the second fundamental form of $\Sigma$ as defined by
     \begin{equation*}
     	\left\langle \Rmnum{2}\left( X,Y\right), Z\right\rangle :=
     	\left\langle \bar{D}_{X}Y, Z\right\rangle =
     	-\left\langle \bar{D}_{X}Z, Y\right\rangle,
     \end{equation*}
    where $X$ and $Y$ are tangent vector fields on $\Sigma$,
    $Z$ is a normal vector field to $\Sigma$, $\bar{D}$ denotes the standard connection on $\mathbb{R}^{n+m}$.
   Further, $\left\langle A, \Rmnum{2}\right\rangle\left( x\right) $ is the normal vector at $x\in\Sigma$ defined by
    \begin{equation*}
    	\left\langle A, \Rmnum{2}\right\rangle \left( x\right) =
         g_{\Sigma}^{ik}g_{\Sigma}^{jl}A_{ij}\Rmnum{2}(\frac{\partial}{\partial x ^{k}},\frac{\partial}{\partial x ^{l}}).
    \end{equation*}
At last, we list the following lemma which extends the arithmetic-geometric mean inequality to the product of two square matrices version:
    \begin{lemma}(Lemma A.1 in \cite{DP})	
   	For $n\in\mathbb{N}$, let $A$ and $B$ be square symmetric matrices of size $n$.
   	Assume that $A$ is positive definite and $B$ is non-negative definite. Then
    	\begin{equation*}
    		\mathrm{det}AB
    		\leq\left(\frac{\mathrm{tr}\left( AB\right)}{n} \right)^{n}.
    	\end{equation*}
    	The equality holds if and only if $AB=\lambda I_{n}$ for some $\lambda\geq0$, 
    	where $I_{n}$ is the identity matrix.
    \end{lemma}

	\section{Proof of Theorem 1.3}

    First, we prove Theorem 1.3 in the special case that $\Sigma$ is connected. By scaling, we assume that 
    \begin{equation} \label{eq:2.1}
    	\int_{\Sigma}\sqrt{\left| \mathrm{div}_{\Sigma}A\right|^{2}+\left| \left\langle A,\Rmnum{2}\right\rangle \right|^{2}}
    	+\int_{\partial\Sigma}\left| A\left( \nu\right) \right| =
    	n \int_{\Sigma}\left( \mathrm{det}A\right) ^{\frac{1}{n-1}}.
    \end{equation}
    Due to that $\Sigma$ is connected and $\eqref{eq:2.1}$, there exists a solution $u:\Sigma\rightarrow\mathbb{R}$ to the following equation 
   	\begin{equation}  
    	\left\{\begin{aligned}
    		&\mathrm{div}_{\Sigma}\left( A\left( \nabla^{\Sigma}u\right) \right) \left( x\right) 
    		=n\left( \mathrm{det}A\left( x\right) \right) ^{\frac{1}{n-1}}
    		-\sqrt{\left| \mathrm{div}_{\Sigma}A\right|^{2}\left( x\right) + \left| \left\langle A,\Rmnum{2}\right\rangle \right|^{2}\left( x\right) },\ \mathrm{in}\  \Sigma\backslash\partial\Sigma,  \\
    		&\left\langle A\left( \nabla^{\Sigma}u\right) \left( x\right),\nu\left( x\right) \right\rangle =\left| A\left( \nu\left( x\right)\right) \right|,\ \mathrm{on}\  \partial\Sigma.
    	\end{aligned} \right.
     \end{equation}
     We define
    \begin{equation*}
    	\begin{split}
        &\Omega:=\left\lbrace x\in\Sigma\backslash\partial\Sigma:
        \left| \nabla^{\Sigma}u\left( x\right) \right|<1\right\rbrace,\\  
    	&U:=\{ \left( x,y\right): x\in\Sigma\backslash\partial\Sigma,
    	y\in T_{x}^{\bot}\Sigma,
    	\left| \nabla^{\Sigma}u\left( x\right) \right|^{2}+\left| y\right| ^{2}<1\},\\  
    	&V:=\left\lbrace \left( x,y\right)\in U: 
        D_{\Sigma}^{2}u\left( x\right)
        -\left\langle \Rmnum{2}\left( x\right), y\right\rangle
        \geq0\right\rbrace,  \\
        \end{split}
    \end{equation*} 
    and a map $\Phi:T^{\perp}(\Sigma\backslash\partial\Sigma)\rightarrow\mathbb{R}^{n+m}$ by
    \begin{equation*}
    	\Phi\left( x, y\right) =\nabla^{\Sigma}u\left( x\right)+y
    \end{equation*} 
    for all $x\in\Sigma\backslash\partial\Sigma$ and 
    $y\in T_{x}^{\bot}\Sigma$.
    Standard elliptic regularity theory implies that the function $u\in C^{2, \gamma}(\Sigma)$ and $\Phi\in C^{1, \gamma}(T^{\perp}(\Sigma\backslash\partial\Sigma))$
    for each $0<\gamma<1 \left( \mathrm{see} \cite{DG}\right)$.
    
 	\begin{lemma}	\label{Le:2.1}
        $\Phi\left( V\right) =B^{n+m}$.
    \end{lemma}
    \begin{proof}
    By the definitions of $U$ and $\Phi$, clearly, $\Phi\left( V\right) \subset B^{n+m}$. It is remained to prove $B^{n+m}\subset\Phi\left( V\right)$.
    Given an arbitrary vector $\xi\in\mathbb{R}^{n+m}$ such that $\left| \xi\right| <1$, we define a function
    $\omega:\Sigma\rightarrow\mathbb{R}$ by 
    $\omega\left( x\right) :=u\left( x\right) -\left\langle x, \xi\right\rangle $.
    Then
    \begin{equation*}
        \nabla^{\Sigma}\omega\left( x\right) =\nabla^{\Sigma}u\left( x\right) - \xi^{\top}
         \mathrm{for}\   x\in\Sigma,
    \end{equation*} 
    where $\xi^{\top}$  is the tangent part of $\xi$ to $\Sigma$. 
    Let $x_{0}$ be a minimum point of $\omega$,
    if $x_{0}\in\partial\Sigma$, then 
    \begin{equation*}
    	\nabla^{\Sigma}\omega\left( x_{0}\right) =\left\langle \nabla^{\Sigma}\omega\left( x_{0}\right), \nu\left( x_{0}\right) \right\rangle\nu\left( x_{0}\right),
    \end{equation*} 
    and
    \begin{equation*}
    	\left\langle \nabla^{\Sigma}\omega\left( x_{0}\right), \nu\left( x_{0}\right) \right\rangle \leq0.
    \end{equation*} 
    Since $A$ is symmetric positive defined and $\left| \xi\right| <1$, according to Cauchy-Schwarz inequality, 
    we have
    \begin{equation*}
       \begin{split}
        0&\geq \left\langle \nabla^{\Sigma}\omega\left(     x_{0}\right), \nu\left( x_{0}\right) \right\rangle
    	\left\langle \nu\left( x_{0}\right), A\left( \nu\left( x_{0}\right) \right) \right\rangle\\
    	&=\left\langle \nabla^{\Sigma}\omega\left( x_{0}\right), A\left( \nu\left( x_{0}\right) \right) \right\rangle\\
    	&=\left\langle \nabla^{\Sigma}u\left( x_{0}\right), A\left( \nu\left( x_{0}\right) \right) \right\rangle
    	-\left\langle\xi^{\top}, A\left( \nu\left( x_{0}\right) \right) \right\rangle\\
    	&=\left| A\left( \nu\left( x_{0}\right) \right)\right|
    	-\left\langle\xi, A\left( \nu\left( x_{0}\right) \right) \right\rangle\\
    	&>0.
      \end{split}
    \end{equation*}
    Thus, $\omega$ must attain its minimum in $\Sigma\backslash\partial\Sigma$, i.e. $x_{0}\in\Sigma\backslash\partial\Sigma$.
    Consequently, we have
    \begin{equation*}
    	\nabla^{\Sigma}\omega\left( x\right) =\nabla^{\Sigma}u\left( x\right) - \xi^{\top}
    	=0.
    \end{equation*} 
     Denote $\xi-\xi^{\top}$ by $y_{0}$, then
     \begin{equation*}
     	\nabla^{\Sigma}u\left( x_{0}\right)+y_{0}=\xi\in B^{n+m}. 
     \end{equation*} 
    Moreover, we obtain
    \begin{equation*}
        D_{\Sigma}^{2} \omega\left( x_{0}\right)
        =D_{\Sigma}^{2}u\left( x_{0}\right)
        -\left\langle \Rmnum{2}\left( x_{0}\right), y_{0}\right\rangle
        \geq0,  
    \end{equation*} 
    Therefore $\left( x_{0}, y_{0}\right)\in V $ and $\Phi\left( x_{0}, y_{0}\right)=\xi$. 
    Thus $B^{n+m}\subset\Phi\left( V\right)$.
    The lemma follows.
    \end{proof} 	
 
    \begin{lemma}	\label{Le:2.2}
    (\cite{Bre}, Lemma 5)
     The Jacobian determinant of $\Phi$ is given by
      \begin{equation*}
      	\mathrm{det}D\Phi\left( x, y\right) 
      	=\mathrm{det}\left( D_{\Sigma}^{2}u\left( x\right)
      	-\left\langle \Rmnum{2}\left( x\right), y\right\rangle \right) 
      \end{equation*} 
    for all $\left( x,y\right) \in T^{\perp}(\Sigma\backslash\partial\Sigma)$.
    \end{lemma}

     \begin{lemma}	\label{Le:2.3}
    	
    	The Jacobian determinant of $\Phi$ satisfies
    	\begin{equation*}
    	   0\leq\mathrm{det}D\Phi\left( x, y\right) 
    		\leq\left( \mathrm{det}A\left( x\right) \right)^{\frac{1}{n-1}} 
    	\end{equation*} 
    	for all $\left( x,y\right) \in V$.
    \end{lemma}

    \begin{proof}
    Given a point $\left( x,y\right) \in V$,
    by Cauchy-Schwarz inequality and 
    $\left| \nabla^{\Sigma}u\left( x\right) \right|^{2}+\left| y\right| ^{2}<1$,
    we have
    \begin{equation*}
    	\begin{split}
    		&-\left\langle \mathrm{div}_{\Sigma}A\left( x\right), \nabla^{\Sigma}u\left( x\right)\right\rangle
    		- \left\langle A\left( x\right) , \left\langle \Rmnum{2}\left( x\right), y\right\rangle\right\rangle  \\
    		=&-\left\langle \mathrm{div}_{\Sigma}A\left( x\right), \nabla^{\Sigma}u\left( x\right)\right\rangle
    		- \left\langle \left\langle A\left( x\right) , \Rmnum{2}\left( x\right)\right\rangle , y\right\rangle\\
    		=&-\left\langle \mathrm{div}_{\Sigma}A\left( x\right)+\left\langle A\left( x\right) , \Rmnum{2}\left( x\right)\right\rangle, \nabla^{\Sigma}u\left( x\right)+y\right\rangle\\
    		\leq&\sqrt{\left| \nabla^{\Sigma}u\left( x\right)\right| ^{2}+\left| y\right| ^{2}}
    		\sqrt{\left| \mathrm{div}_{\Sigma}A\right| ^{2}\left( x\right)+\left| \left\langle A, \Rmnum{2}\right\rangle\right| ^{2}\left( x\right) }\\
    		\leq&\sqrt{\left| \mathrm{div}_{\Sigma}A\right| ^{2}\left( x\right)+\left| \left\langle A, \Rmnum{2}\right\rangle\right| ^{2}\left( x\right) }.
    	\end{split}
    \end{equation*}
    Note that 	
    \begin{equation*}
        \mathrm{div}_{\Sigma}\left( A\left( \nabla^{\Sigma}u\right) \right)
        =\left\langle \mathrm{div}_{\Sigma}A, \nabla^{\Sigma}u\right\rangle 
        + \left\langle A, D_{\Sigma}^{2}u\right\rangle.
    \end{equation*}
    Thus, by the equation of $u$, we obtain
    \begin{equation*}
    	\begin{split}
    		&\left\langle A\left( x\right), D_{\Sigma}^{2}u\left( x\right)-\left\langle \Rmnum{2}\left( x\right), y\right\rangle\right\rangle \\
    		=&\mathrm{div}_{\Sigma}\left( A\left( \nabla^{\Sigma}u\right) \right)\left( x\right) 
    		-\left\langle \mathrm{div}_{\Sigma}A\left( x\right), \nabla^{\Sigma}u\left( x\right)\right\rangle
    		-\left\langle A\left( x\right) , \left\langle \Rmnum{2}\left( x\right), y\right\rangle\right\rangle \\
    		=&n\left( \mathrm{det}A\left( x\right) \right) ^{\frac{1}{n-1}}
    		-\sqrt{\left| \mathrm{div}_{\Sigma}A\right| ^{2}\left( x\right)+\left| \left\langle A, \Rmnum{2}\right\rangle\right| ^{2}\left(x\right)}	\\
    		&-\left\langle \mathrm{div}_{\Sigma}A\left( x\right), \nabla^{\Sigma}u\left( x\right)\right\rangle
    		- \left\langle A\left( x\right), \left\langle \Rmnum{2}\left( x\right), y\right\rangle\right\rangle \\
    		\leq& n\left( \mathrm{det}A\left( x\right) \right) ^{\frac{1}{n-1}}.
    	\end{split}
    \end{equation*}
    Since $A>0$ and $D_{\Sigma}^{2}u\left( x\right)
    -\left\langle \Rmnum{2}\left( x\right), y\right\rangle\geq0$ for all $\left( x,y\right) \in V$,
    by Lemma 2.1, we have 
    \begin{equation*}
    	\begin{split}
    		&\mathrm{det}\left( D_{\Sigma}^{2}u\left( x\right)
    		-\left\langle \Rmnum{2}\left( x\right), y\right\rangle\right) \\
    		=&\frac{1}{\mathrm{det}A\left( x\right)}
    		\mathrm{det}\left[ A\left( x\right) \circ\left( D_{\Sigma}^{2}u\left( x\right)
    		-\left\langle \Rmnum{2}\left( x\right), y\right\rangle\right)\right] \\ 
    		\leq&\frac{1}{\mathrm{det}A\left( x\right)}
    		\left( \frac{\left\langle A\left( x\right),  D_{\Sigma}^{2}u\left( x\right)
    		-\left\langle \Rmnum{2}\left( x\right), y\right\rangle\right\rangle }{n}\right) ^{n}\\
    	   	\leq&\left( \mathrm{det}A\left( x\right)\right) 
    	   	^{\frac{1}{n-1}}.\\
    	\end{split}
    \end{equation*}
    By Lemma 3.2, this lemma follows.
    \end{proof}
\emph{Proof of Theorem 1.3.} Given a constant $\sigma$ such that $0\leq\sigma<1$,
    by area formula and Lemma 3.3, we have
    \begin{equation*}
    	\begin{split}
    		&\left| B^{n+m}\right|\left( 1-\sigma^{n+m}\right)   \\
    		=&\int_{\left\lbrace \xi\in\mathbb{R}^{n+m}:\sigma^{2}<\left| \xi\right|^{2}<1\right\rbrace }1 d\xi \\ 
    		\leq& \int_{\Omega}\left( 
    		\int_{\left\lbrace y\in T_{x}^{\bot}\Sigma:\sigma^{2}<\left| \Phi\left( x, y\right) \right|^{2}<1\right\rbrace } 
    		\left| \mathrm{det}D\Phi\left( x, y\right) \right| 1_{A}\left( x, y\right)
    		dy\right) d\mathrm{vol}\left( x\right) \\
    		\leq& \int_{\Omega}\left( 
    		\int_{\left\lbrace y\in T_{x}^{\bot}\Sigma:\sigma^{2}<\left| \nabla^{\Sigma}u\left( x\right) \right|^{2}+\left| y\right| ^{2}<1\right\rbrace } 
    		\left( \mathrm{det}A\left( x\right)\right) 
    		^{\frac{1}{n-1}}dy\right) d\mathrm{vol}\left( x\right) \\
    		\leq&\frac{m}{2}\left| B^{m}\right|\left( 1-\sigma^{2}\right) 
    		\int_{\Omega} \left( \mathrm{det}A\right) 
    		^{\frac{1}{n-1}} \\.
    	\end{split}
    \end{equation*}
    Dividing both side by $1-\sigma$ and letting $\sigma\rightarrow1^{-}$, we have
    \begin{equation}
    	\left( n+m\right)\left| B^{n+m}\right|\leq m\left| B^{m}\right| 
    	\int_{\Omega} \left( \mathrm{det}A\right) 
    	^{\frac{1}{n-1}}\leq
        m\left| B^{m}\right| 
    	\int_{\Sigma} \left( \mathrm{det}A\right) 
    	^{\frac{1}{n-1}}.
    \end{equation}
    From the scaling assumption $\eqref{eq:2.1}$, we obtain
    \begin{equation*}
   	    \begin{split}
   	    	&\int_{\Sigma}\sqrt{\left| \mathrm{div}_{\Sigma}A\right| ^{2}+\left| \left\langle A, \Rmnum{2}\right\rangle\right| ^{2}}
   	    	+\int_{\partial\Sigma}\left| A\left( \nu\right) \right| \\
   	    	=&n\int_{\Sigma} \left( \mathrm{det}A\right) 
   	    	^{\frac{1}{n-1}}\\ 
   	    	=&n\left( \int_{\Sigma} \left( \mathrm{det}A\right) 
   	    	^{\frac{1}{n-1}}\right) ^{\frac{1}{n}}
   	    	\left( \int_{\Sigma} \left( \mathrm{det}A\right) 
   	    	^{\frac{1}{n-1}}\right) ^{\frac{n-1}{n}}\\
   	    	\geq &n\left[ \frac{\left( n+m\right)\left| B^{n+m}\right|}{m\left| B^{m}\right|}\right] 
   	    	^{\frac{1}{n}}
   	    	\left( \int_{\Sigma} \left( \mathrm{det}A\right) 
   	    	^{\frac{1}{n-1}}\right) ^{\frac{n-1}{n}}.	
   	    \end{split}
    \end{equation*}
    Thus, the special case when $\Sigma$ is connected to Theorem 1.3 has been proved. If $\Sigma$ is disconnected,
    we can conclude that
    \begin{equation*}
		\int_{\Sigma}\sqrt{\left| \mathrm{div}_{\Sigma}A\right| ^{2} +\left| \left\langle A, \Rmnum{2}\right\rangle\right| ^{2}}
		+\int_{\partial\Sigma}\left| A\left( \nu\right) \right| 
		> n\left[ \frac{\left( n+m\right)\left| B^{n+m}\right|}{m\left| B^{m}\right|}\right] 
		^{\frac{1}{n}}
		\left( \int_{\Sigma} \left( \mathrm{det}A\right) 
	  	^{\frac{1}{n-1}}\right) ^{\frac{n-1}{n}}
    \end{equation*}    
    from the inequality 
    $a^{\frac{n-1}{n}}+b^{\frac{n-1}{n}}
    >\left( a+b\right) ^{\frac{n-1}{n}}$
    for $a, b>0$.
    This completes the proof of Theorem 1.3.\hfill$\Box$\\

\section{Proof of Theorem 1.4}
Suppose that $\Sigma^{n}$ is a compact $n$-dimensional submanifold of $\mathbb{R}^{n+2}$ with smooth boundary $\partial\Sigma$ (possibly  $\partial\Sigma=\varnothing$),  
and $A$ is a smooth symmetric uniformly positive define $\left( 0,2\right)$-tensor field on $\Sigma$ satisfying
\begin{equation*}
	\int_{\Sigma}\sqrt{\left| \mathrm{div}_{\Sigma}A\right|^{2}+\left| \left\langle A,\Rmnum{2}\right\rangle \right|^{2}}
	+\int_{\partial\Sigma}\left| A\left( \nu\right) \right| =
	n\left| B^{n}\right|^{\frac{1}{n}}\left( \int_{\Sigma}\left( \mathrm{det}A\right)^{\frac{1}{n-1}}\right) ^{\frac{n-1}{n}}.
\end{equation*}
We conclude that $\Sigma$ is connected from the last part of the proof of Theorem 1.3. 

In the sense of a difference of a positive factor of $A$, we may assume that 	
\begin{equation}
		\int_{\Sigma}\sqrt{\left| \mathrm{div}_{\Sigma}A\right|^{2}+\left| \left\langle A,\Rmnum{2}\right\rangle \right|^{2}}
		+\int_{\partial\Sigma}\left| A\left( \nu\right) \right| =
		n \int_{\Sigma}\left( \mathrm{det}A\right) ^{\frac{1}{n-1}}
	\end{equation}
and 
	\begin{equation}
\int_{\Sigma}	\left( \mathrm{det}A\right) ^{\frac{1}{n-1}}=|B^{n}|.
		\end{equation}
Let $u$, $\Omega$, $U$, $V$ and $\Phi:U\rightarrow \mathbb{R}^{n+2}$	be defined as in Section 3. Combining
(3.3) and (4.2), we have 
    \begin{equation*}
	\left( n+2\right)\left| B^{n+2}\right|\leq 2\left| B^{2}\right| 
	\int_{\Omega} \left( \mathrm{det}A\right) 
	^{\frac{1}{n-1}}\leq
	2\left| B^{2}\right| 
	\int_{\Sigma} \left( \mathrm{det}A\right) 
	^{\frac{1}{n-1}}=\left( n+2\right)\left| B^{n+2}\right|.
\end{equation*}
Thus, we conclude the following lemma.
\begin{lemma}
$\Omega$ is dense in $\Sigma$. Moreover, $\Omega=\Sigma\backslash\partial \Sigma$ and $|\nabla^{\Sigma}u(x)|=1$ for all $x\in \partial \Sigma$.
\end{lemma}

\begin{lemma}
	Suppose that $\bar{x}\in\Omega$ and $\bar{y}\in T^{\perp}_{\bar{x}}\Sigma$ satisfying $\left| \nabla^{\Sigma}u\left( \bar{x}\right) \right|^{2}+\left| \bar{y}\right| ^{2}=1$. If the smallest eigenvalue of $D_{\Sigma}^{2}u\left( \bar{x}\right)
	-\left\langle \Rmnum{2}\left( \bar{x}\right), \bar{y}\right\rangle$ is negative or $\mathrm{det}\left(D_{\Sigma}^{2}u\left( \bar{x}\right)
	-\left\langle \Rmnum{2}\left( \bar{x}\right), \bar{y}\right\rangle\right)<  \left( \mathrm{det}A(\bar{x})\right) ^{\frac{1}{n-1}} $ when $D_{\Sigma}^{2}u\left( \bar{x}\right)
	-\left\langle \Rmnum{2}\left( \bar{x}\right), \bar{y}\right\rangle\geq 0$, then there exists $\epsilon\in (0,1)$, $\delta\in (0,1)$, an open neighborhood $W$ of $\bar{x}$ in $\Sigma\backslash\partial \Sigma$ and $Z:=\{ \left( x,y\right): x\in W,
	y\in T_{x}^{\bot}\Sigma,
	1-\delta<\left| \nabla^{\Sigma}u\left( x\right) \right|^{2}+\left| y\right| ^{2}<1+\delta\}$ such that 
	\begin{equation*}
	\mathrm{det}D\Phi\left( x, y\right) 
	\leq (1-\epsilon)\left( \mathrm{det}A\left( x\right) \right)^{\frac{1}{n-1}} 
	\end{equation*}
	for all $(x,y)\in Z\cap V$.
	\end{lemma}
	\begin{proof}
	Clearly, this lemma follows from $u\in C^{2}(\Sigma)$ and Lemma 3.2.
	\end{proof}
	
	\begin{lemma}
	Suppose that $\bar{x}\in\Omega$ and $\bar{y}\in T^{\perp}_{\bar{x}}\Sigma$ satisfying $\left| \nabla^{\Sigma}u\left( \bar{x}\right) \right|^{2}+\left| \bar{y}\right| ^{2}=1$. Then $D_{\Sigma}^{2}u\left( \bar{x}\right)
	-\left\langle \Rmnum{2}\left( \bar{x}\right), \bar{y}\right\rangle\geq 0$ and 	$\mathrm{det}\left(D_{\Sigma}^{2}u\left( \bar{x}\right)
	-\left\langle \Rmnum{2}\left( \bar{x}\right), \bar{y}\right\rangle\right)=  \left( \mathrm{det}A(\bar{x})\right) ^{\frac{1}{n-1}}$.	
	\end{lemma}
	\begin{proof}
	We argue by contradiction. Assume that the smallest eigenvalue of $D_{\Sigma}^{2}u\left( \bar{x}\right)
	-\left\langle \Rmnum{2}\left( \bar{x}\right), \bar{y}\right\rangle$ is negative or $\mathrm{det}\left(D_{\Sigma}^{2}u\left( \bar{x}\right)
	-\left\langle \Rmnum{2}\left( \bar{x}\right), \bar{y}\right\rangle\right)<  \left( \mathrm{det}A(\bar{x})\right) ^{\frac{1}{n-1}} $ when $D_{\Sigma}^{2}u\left( \bar{x}\right)
	-\left\langle \Rmnum{2}\left( \bar{x}\right), \bar{y}\right\rangle\geq 0$. By Lemma 4.2, there exists $\epsilon\in (0,1)$, $\delta\in (0,1)$, an open neighborhood $W$ of $\bar{x}$ in $\Sigma\backslash\partial \Sigma$ and $Z:=\{ \left( x,y\right): x\in W,
	y\in T_{x}^{\bot}\Sigma,
	1-\delta<\left| \nabla^{\Sigma}u\left( x\right) \right|^{2}+\left| y\right| ^{2}<1+\delta\}$ such that 
	\begin{equation*}
		\mathrm{det}D\Phi\left( x, y\right) 
		\leq (1-\epsilon)\left( \mathrm{det}A\left( x\right) \right)^{\frac{1}{n-1}} 
	\end{equation*}
	for all $(x,y)\in Z\cap V$. Using Lemma 3.3, we deduce that 
		\begin{equation*}
		0\leq\mathrm{det}D\Phi\left( x, y\right) 
		\leq (1-\epsilon\cdot 1_{Z}(x, y))\left( \mathrm{det}A\left( x\right) \right)^{\frac{1}{n-1}} 
	\end{equation*}
		for all $(x,y)\in V$.
Similar to the proof of Theorem 1.3, given a constant $\sigma$ such that $1-\delta<\sigma^{2}<1$, we have	
\begin{align*}
			&\left| B^{n+2}\right|\left( 1-\sigma^{n+2}\right)   \\
			=&\int_{\left\lbrace \xi\in\mathbb{R}^{n+2}:\sigma^{2}<\left| \xi\right|^{2}<1\right\rbrace }1 d\xi \\ 
			\leq& \int_{\Omega}\left( 
			\int_{\left\lbrace y\in T_{x}^{\bot}\Sigma:\sigma^{2}<\left| \Phi\left( x, y\right) \right|^{2}<1\right\rbrace } 
			\left| \mathrm{det}D\Phi\left( x, y\right) \right| 1_{A}\left( x, y\right)
			dy\right) d\mathrm{vol}\left( x\right) \\
			\leq& \int_{\Omega}\left( 
\int_{\left\lbrace y\in T_{x}^{\bot}\Sigma:\sigma^{2}<\left| \nabla^{\Sigma}u\left( x\right) \right|^{2}+\left| y\right| ^{2}<1\right\rbrace }
(1-\epsilon\cdot 1_{Z}(x, y))
\left( \mathrm{det}A\left( x\right) \right)^{\frac{1}{n-1}} dy\right) d\mathrm{vol}\left( x\right) \\
			=&| B^{2}|\int_{\Omega}\left[(1-\left| \nabla^{\Sigma}u\left( x\right) \right|^{2})-(\sigma^{2}-\left| \nabla^{\Sigma}u\left( x\right) \right|^{2})_{+}   \right]\left( \mathrm{det}A\left( x\right) \right)^{\frac{1}{n-1}}d\mathrm{vol}\left( x\right)\\
			&-\epsilon| B^{2}|\int_{\Omega\cap W}\left[(1-\left| \nabla^{\Sigma}u\left( x\right) \right|^{2})-(\sigma^{2}-\left| \nabla^{\Sigma}u\left( x\right) \right|^{2})_{+}   \right]\left( \mathrm{det}A\left( x\right) \right)^{\frac{1}{n-1}}d\mathrm{vol}\left( x\right)\\
			=&| B^{2}|\int_{\Omega}\left[(1-\left| \nabla^{\Sigma}u\left( x\right) \right|^{2})-(\sigma^{2}-\left| \nabla^{\Sigma}u\left( x\right) \right|^{2})_{+}   \right]\left(1-\epsilon 1_{W} (x) \right)\left( \mathrm{det}A\left( x\right) \right)^{\frac{1}{n-1}}d\mathrm{vol}\left( x\right)\\
			\leq&\left| B^{2}\right|\left( 1-\sigma^{2}\right) 
			\int_{\Omega} \left( \mathrm{det}A\right) 
			^{\frac{1}{n-1}}-\epsilon\left| B^{2}\right|\left( 1-\sigma^{2}\right) 
			\int_{\Omega\cap W} \left( \mathrm{det}A\right) 
			^{\frac{1}{n-1}} .
\end{align*}
	 Dividing both side by $1-\sigma$ and letting $\sigma\rightarrow1^{-}$, we have
	\begin{equation*}
				\begin{split}
		\left( n+2\right)\left| B^{n+2}\right|&\leq 2\left| B^{2}\right| 
		\int_{\Omega} \left( \mathrm{det}A\right) 
		^{\frac{1}{n-1}}  -2\epsilon\left| B^{2}\right|\int_{\Omega\cap W} \left( \mathrm{det}A\right) 
		^{\frac{1}{n-1}}\\
		&\leq
		2\left| B^{2}\right| 
		\int_{\Sigma} \left( \mathrm{det}A\right) 
		^{\frac{1}{n-1}} -2\epsilon\left| B^{2}\right|\int_{\Omega\cap W} \left( \mathrm{det}A\right) 
		^{\frac{1}{n-1}}\\
		&=\left( n+2\right)\left| B^{n+2}\right|-2\epsilon\left| B^{2}\right|\int_{\Omega\cap W} \left( \mathrm{det}A\right) 
		^{\frac{1}{n-1}}.
	\end{split}
	\end{equation*}
Since $\Omega\cap W$ is a nonempty open set in $\Sigma$ and $A>0$, it's a contradiction. Therefore, the lemma follows.
			\end{proof}

	\begin{lemma}
	$\Rmnum{2}\left( x\right)\equiv0$, $\mathrm{div}_{\Sigma}A(x)\equiv 0$,  $D_{\Sigma}^{2}u\left( x\right)
	> 0$ ,	$\mathrm{det}\left(D_{\Sigma}^{2}u\left( x\right)
\right)=  \left( \mathrm{det}A(x)\right) ^{\frac{1}{n-1}}$ and $A(x)=\mathrm{cof} D^{2}_{\Sigma}u(x)$	  for all $x\in \Sigma$. Moreover, $u\in C^{\infty}(\Sigma)$.
	\end{lemma}
	\begin{proof}
	Given a point $\bar{x}\in \Omega$, since  $\left| \nabla^{\Sigma}u\left( \bar{x}\right) \right|^{2}<1$, we have  $\left| \nabla^{\Sigma}u\left( \bar{x}\right) \right|^{2}+\left| \bar{y}\right| ^{2}=1$ for all $\bar{y}\in\{y\in T_{\bar{x}}^{\perp}\Sigma :\left| y\right| ^{2}=1-\left| \nabla^{\Sigma}u\left( \bar{x}\right) \right|^{2}\}$. By Lemma 4.3, we obtain $D_{\Sigma}^{2}u\left( \bar{x}\right)
	-\left\langle \Rmnum{2}\left( \bar{x}\right), \bar{y}\right\rangle\geq 0$ and
	\begin{equation}
	 \mathrm{det}\left(D_{\Sigma}^{2}u\left( \bar{x}\right)
	-\left\langle \Rmnum{2}\left( \bar{x}\right), \bar{y}\right\rangle\right)=  \left( \mathrm{det}A(\bar{x})\right) ^{\frac{1}{n-1}}.	
	\end{equation}
	Similar to the proof of Lemma 3.3, since $\left| \nabla^{\Sigma}u\left( \bar{x}\right) \right|^{2}+\left| \bar{y}\right| ^{2}=1$, $A(\bar{x})>0$ and $D_{\Sigma}^{2}u\left( \bar{x}\right)
	-\left\langle \Rmnum{2}\left( \bar{x}\right), \bar{y}\right\rangle\geq 0$, we have 
	\begin{equation*}
		\mathrm{det}\left( D_{\Sigma}^{2}u\left( \bar{x}\right)
		-\left\langle \Rmnum{2}\left(\bar{ x}\right), \bar{y}\right\rangle\right)\leq\left( \mathrm{det}A\left( \bar{x}\right)\right) 
		^{\frac{1}{n-1}}
	\end{equation*}
	with the equality holds. Thus there exists  $\mu\geq 0$ and $\nu \geq 0$ such that
\begin{align}
			&\mathrm{div}_{\Sigma}A(\bar{x})=-\mu\nabla^{\Sigma}u\left( \bar{x}\right) ,\\  
			&\left\langle A(\bar{x}), \Rmnum{2}\left( \bar{x}\right)\right\rangle=-\mu y,\\  
			&A\left( \bar{x}\right) \circ\left( D_{\Sigma}^{2}u\left( \bar{x}\right)
			-\left\langle \Rmnum{2}\left( \bar{x}\right), \bar{y}\right\rangle\right)=\nu g_{\Sigma}(\bar{x}).
\end{align}
We conclude that $\nu=\left( \mathrm{det}A(\bar{x})\right) ^{\frac{1}{n-1}}$ and 
		\begin{equation}
 D_{\Sigma}^{2}u\left( \bar{x}\right)
-\left\langle \Rmnum{2}\left( \bar{x}\right), \bar{y}\right\rangle= \left( \mathrm{det}A(\bar{x})\right) ^{\frac{1}{n-1}}A^{-1}(\bar{x})
	\end{equation}
	from (4.3) and (4.6). Replacing $\bar{y}$ by $-\bar{y}$ gives
			\begin{equation}
		D_{\Sigma}^{2}u\left( \bar{x}\right)
		-\left\langle \Rmnum{2}\left( \bar{x}\right), -\bar{y}\right\rangle= \left( \mathrm{det}A(\bar{x})\right) ^{\frac{1}{n-1}}A^{-1}(\bar{x}).
	\end{equation}
	Combining (4.7) and (4.8), we obtain $D_{\Sigma}^{2}u\left( \bar{x}\right)
	= \left( \mathrm{det}A(\bar{x})\right) ^{\frac{1}{n-1}}A^{-1}(\bar{x})$ and $\left\langle \Rmnum{2}\left( \bar{x}\right), \bar{y}\right\rangle=0$ for all $\bar{x}\in \Omega$ and $\bar{y}\in\{y\in T_{\bar{x}}^{\perp}\Sigma :\left| y\right| ^{2}=1-\left| \nabla^{\Sigma}u\left( \bar{x}\right) \right|^{2} \}$. Due to the range of $\bar{y}$ and $u\in C^{2}(\Sigma)$, by Lemma 4.1, we have 
			\begin{equation}	
 D_{\Sigma}^{2}u\left( x\right)
= \left( \mathrm{det}A(x)\right) ^{\frac{1}{n-1}}A^{-1}(x)>0
		\end{equation}
and 	
				\begin{equation}	
	\Rmnum{2}\left( x\right)\equiv0 
			\end{equation}
	for all $x\in \Sigma$. By substituting (4.10) into (4.3), by Lemma 4.1, we obtain
			\begin{equation*}		
\mathrm{det}\left(D_{\Sigma}^{2}u\left( x\right)
	\right)=  \left( \mathrm{det}A(x)\right) ^{\frac{1}{n-1}}  
				\end{equation*}
for all $x\in \Sigma$.	By substituting (4.10) into (4.5), we obtain $\mu=0$ which concludes 
				\begin{equation}	
\mathrm{div}_{\Sigma}A(x)=0
\end{equation}
	for all $x\in \Sigma$ from (4.4) and Lemma 4.1. Combining (4.3), (4.9) and (4.10), we obtain 
					\begin{equation*}	
A=\mathrm{cof} D^{2}_{\Sigma}u.
	\end{equation*}
	By substituting (4.10) and (4.11) into (3.2), we obtain
		\begin{equation*}  
		\left\{\begin{aligned}
			&\mathrm{div}_{\Sigma}\left( A\left( \nabla^{\Sigma}u\right) \right) \left( x\right) 
			=n\left( \mathrm{det}A\left( x\right) \right) ^{\frac{1}{n-1}}
			,\ \mathrm{in}\  \Sigma\backslash\partial\Sigma,  \\
			&\left\langle A\left( \nabla^{\Sigma}u\right) \left( x\right),\nu\left( x\right) \right\rangle =\left| A\left( \nu\left( x\right)\right) \right|,\ \mathrm{on}\  \partial\Sigma.
		\end{aligned} \right.
	\end{equation*}
Since $A$ is smooth, standard elliptic regularity theory implies $u\in C^{\infty}(\Sigma)$.
These complete the proof.
		\end{proof}
 By Lemma 4.1, $\nabla^{\Sigma}u$ maps $\Sigma$ into the closure of $B^{m}\subset \mathbb{R}^{n}$ and $\nabla^{\Sigma}u(\Omega)=B^{m}$. Since $D^{2}_{\Sigma}u>0$, $\nabla^{\Sigma}u$ is a local homeomorphism. We conclude that
\begin{lemma}
$\nabla^{\Sigma}u(\Sigma)=\overline{B^{m}}$.
\end{lemma}
\emph{Proof of Theorem 1.4.}
The necessity follows from Lemma 4.4 and Lemma 4.5. The sufficiency follows from Theorem 1.2.

	\hfill$\Box$\\

\end{document}